\documentclass[11pt]{article}
\usepackage{fullpage}
\usepackage{multicol, float}
\usepackage{lineno,hyperref}
\usepackage{verbatim, euscript}
\usepackage{bbm,amssymb}
\usepackage{amsmath, xcolor}
\usepackage{amsthm}
\usepackage{mathtools}
\mathtoolsset{showonlyrefs}
\usepackage{tikz}

\theoremstyle{plain}
\newtheorem{thm}{Theorem}

\newtheorem{lem}{Lemma}[section]
  
\newtheorem{cor}[lem]{Corollary}
\newtheorem{defi}[lem]{Definition}
\newtheorem{rem}[lem]{Remark}

\newtheorem{prop}[lem]{Proposition}

\newcommand{\eps}{\varepsilon}

\newcommand{\PP}{\mathbb{P}}
 
\newcommand{\N}{\mathbb{N}}
\newcommand{\Z}{\mathbb{Z}}
\newcommand{\R}{\mathbb{R}}
\newcommand{\Q}{\mathbb{Q}}
\newcommand{\V}{\text{Var}}

\newcommand{\Dc}{\mathcal{D}}

\newcommand{\EE}[1]{\mathbb{E} \left[ #1 \right]}
\newcommand{\VV}[1]{\V \left( #1 \right)}
\newcommand{\pp}[1]{\mathbb{P} \left( #1 \right)}

\newcommand{\norm}[1]{\lVert #1 \rVert_{\infty}}

\definecolor{darkgreen}      {cmyk}{0.90,0.00,0.90,0.10}
\newcommand{\bluenew}[1]{\textcolor{black}{#1}}

\title{Ornstein-Uhlenbeck fluctuations for the line counting process of the ancestral selection graph
}
\author{Florin Boenkost\thanks{florin.boenkost@univie.ac.at}, Anna-Lena Weinel\thanks{anna-lena.weinel@outlook.de}}
\date{\today}

\begin{document}
	\maketitle
	\begin{abstract}
		For the Moran model with strong or moderate selection we prove that the fluctuations around the deterministic limit of the line counting process of the ancestral selection graph converge to an Ornstein-Uhlenbeck process. To this purpose we provide an extension of a functional limit theorem by Ethier and Kurtz \cite{EthierKurtz1986}. This result and a small adaptation of our arguments can also be used to obtain the scaling limit for the fluctuations of certain logistic branching processes.\\\\
		
		
	\end{abstract}
	

	\section{Introduction}
	A prominent task in mathematical population genetics is to understand the random genealogy of a sample consisting of individuals taken from a large population. Pioneered by Kingman \cite{Kingman1982} in 1982, the Kingman coalescent emerges as the limiting genealogy, roughly speaking, whenever the population evolves neutrally, no individual has the ability to impact large fractions of the total population, and the total population size remains constant. Kingman's coalescent is then envisioned as a random tree, where each leaf is associated with an individual from the sample. The tree is traversed from the leaves to the root, and when two branches coalesce, the associated individuals now share a common ancestor. However, it is important to note that Kingman's coalescent is not suitable in the presence of selection.   
	
	In order to describe the genealogy of a population experiencing selection in a traceable way, Krone and Neuhauser in 1997 \cite{Neuhauser1997genealogy} introduced the so-called \emph{ancestral selection graph} (ASG) for the Moran model with directional selection. Unlike Kingman's coalescent, the ASG is allowed to branch, resulting in a branching and coalescing structure. Each individual may now have more than one ancestor, which is why they are usually called \emph{potential ancestors}. A precise definition and construction are given in Section~\ref{Sec:Main}. A key quantity of interest is the line counting process of the ancestral selection graph, which counts the number of distinct potential ancestors of a sample backwards in time. 
	
	Since the seminal work of Krone and Neuhauser, a wide variety of generalisations of their original ASG have been proposed. These extensions incorporate various evolutionary forces, i.e. mutation, recombination and more general selection mechanisms, for a general overview we refer to the review on recent progress in this direction by Baake and Wakolbinger \cite{BaakeWakolbinger2018} and the references provided therein. We just give a brief overview here to present the variety of results obtained, demonstrating the importance of the line-counting processes. 
	
	Generally speaking, the significance of these ancestral processes lies in the fact that the forward in time frequency process is dual to the line counting process with respect to a duality function \cite{Jansen2014}. Common types of duality encountered in mathematical population genetics are moment- and hypergeometric dualities. These dualities allow one to derive results for the forward in time frequency process by instead considering the respective backwards in time processes. Dualities are among the reasons why the ancestral selection graph has been proven applicable to a wide variety of models and questions.
	
	For instance, Pokalyuk and Pfaffelhuber \cite{PP} derive the probability of fixation and the time to fixation using the line counting process of the ancestral selection graph for a Moran model with weak selection, without relying on the commonly used diffusion approximation.  The asymptotic probability of fixation, utilizing the hypergeometric duality relation between the frequency process and the line counting process of the ASG, was for example also applied in \cite{BoeGoPoWa1} for Cannings models and in \cite{BBDP23} for a Moran model with highly skewed offspring. In all cases it is crucial to understand the stationary distribution of the respective line counting process. 
	
	The ancestral selection graph can also be extended to settings including mutations, leading to the pruned lookdown ancestral selection graph as in \cite{Baake2016,Cordero2017common,Lenz2015}. In these papers the authors express the type distribution of the common ancestor of the population by means of the line counting process of the pruned lookdown ancestral selection graph for the $\Lambda$-Wright-Fisher model with selection and mutation. The line counting process for this class of models is further studied in \cite{Cordero2019}, where criteria are given for the line counting process to be positive recurrent and to admit a stationary distribution. More general selection mechanisms have been considered in \cite{Baake2022,Cordero2022}, leading to other types of dualities. However, the general principle remains the same. Key quantities of interest for the forward in time frequency process can be expressed in terms of a line counting process in stationarity. 

    In particular, in the case of the Moran model with selection, the ASG is in pathwise duality with the forward-in-time frequency process. The frequency of potential ancestors at the time a beneficial mutation arose determines the probability that a given individual carries the beneficial allele. As $N \to \infty$, this frequency (of potential ancestors) satisfies a law of large numbers, making it natural to study its fluctuations, which converge to an Ornstein-Uhlenbeck process. To our knowledge, no duality result directly involving these fluctuations is currently available. In a related but different direction, \cite{Favero2025} recently established an asymptotic duality between a typed ASG line-counting process and a continuous-state branching process with immigration, 
    in the context of a multitype Wright-Fisher diffusion near fixation of the fittest type, as the corresponding selection coefficient tends to infinity.
    
    Our main objective is to establish a functional limit theorem for the fluctuations of the line counting process of the ancestral selection graph around its expectation. For all strengths of selection (except weak selection) the fluctuations are Gaussian and the rescaled process converges to an Ornstein-Uhlenbeck process, see Theorem \ref{thm:Moran}.
    
	The line counting process studied here exhibits close similarities with the so-called logistic branching processes \cite{Lambert2005logistic}. More precisely, if one extends the classical notion of a logistic branching process to allow for a frequency dependent branching rate, the line counting process fits into that framework. A logistic branching process is an extension of classical branching processes, where individuals compete with each other, mimicking the behaviour of classical deterministic logistic growth models, in particular they might have a centre of attraction, meaning that there exists some point $n_0\in \N$ such that the process in expectation moves towards $n_0$. For a precise definition of a logistic branching process we refer to Section \ref{sec:logisitic_branching}. The similarities between the logistic branching process and the line counting process suggests that the same limiting objects should arise, which we prove in Theorem~\ref{thm:logistic_branching}. In Theorem~\ref{thm:logistic_branching} we study a class of logistic branching processes which admit a centre of attraction and prove that the fluctuations are Gaussian. However it is still an open question which are the most general birth and death processes showing the same limiting behaviour.
	
    All these results rely on generator calculations and make use of a proposition based on results in Ethier and Kurtz~\cite{EthierKurtz1986}, regarding the convergence of Markovian processes, see Proposition~\ref{prop:Convergence Gen}. This convergence result is particularly well suited for the applications here, however we state it in a broad way, making it possibly useful for other applications as well. Indeed, this proposition was applied in recent work \cite{Igelbrink2025} to prove Ornstein-Uhlenbeck fluctuations for the frequency of a beneficial allele in a Moran model with selection and deleterious mutations, in a parameter regime that guarantees a metatstable behaviour before the beneficial allele's extinction.
	
	This paper is organised as follows. In Section~\ref{Sec:Main} we first introduce the Moran model and the corresponding ancestral selection graph, as well as the line counting process and later introduce the logistic branching process. Afterwards, we state our main theorems, first the convergence results for the logistic branching process, then the one for the line counting process and then the general convergence result on Markov processes in Proposition~\ref{prop:Convergence Gen}.
    The proofs are carried out in Section~\ref{sec:Proofs} and note that the proof of Proposition~\ref{prop:Convergence Gen} is independent of the other results.
    Parts of the results presented here were first obtained in the second author's thesis \cite{Weinel2022}.
	
	\section{Model and Main Results} \label{Sec:Main}
	\subsection{The Moran model and the ancestral selection graph}
	Throughout our work, we consider the classical two-type Moran model \cite{Moran1958} in continuous-time with directional selection, constant population size $N$, which is often envisioned as a graph as depicted in Figure~\ref{fig:Moran}, also see \cite[Chapter 6]{DurretBook}. The $N$ individuals are labelled by $[N]:=\{1,\dots,N\}$, each individual $i\in [N]$ is represented as a line from left (past) to the right (future). For each individual $i \in [N]$ we consider a Poisson Point process $P^{i}$ (governing the neutral reproduction) with rate $\frac{\gamma}{2}>0$ and assume $(P^i, i \in [N])$ to be independent. At each point $t^{i}$ of $P^{i}$ individual $i$ reproduces. At reproduction individual $i$ samples an uniform individual $j\in [N]$ (possibly $i$ itself) and puts its offspring there and individual $j$ inherits the type of $i$. Consequently, we draw an arrow pointing from $i$ to $j$ at time $t^{i}$ in the graphical construction. This leads to the classical graphical construction of the neutral Moran model.
	\begin{figure}[h]
		\centering	\begin{tikzpicture}[x=1mm,scale=1.2,>=stealth]
			\foreach \x in {0,1,2,3,4} \draw[-] (-40,\x) -- (40,\x);
			\foreach \x in {1,2,3,4,5} \node[left] at (-40,\x-1) {\x};
			\draw [->,thick] (-30,2) -- (-30,3);
			\draw [->,thick] (-25,2) -- (-25,1);
			\draw [->,thick] (-17,0) -- (-17,4);
			\draw [->,thick] (-7,4) -- (-7,2);
			\draw [->,thick] (0,0) -- (0,1);
			\draw [->,thick] (12,1) -- (12,2);
			\draw [->,thick] (23,3) -- (23,4);
			\draw [->,dashed,thick] (-35,4) -- (-35,2);
			\draw [->,dashed,thick] (15,0) -- (15,2);
			\draw[-,thick, blue] (-40,2) -- (-35,2);
			\draw[-,blue, thick] (-40,2) -- (-30,2);
			\draw[-,blue, thick] (-30,3) -- (-25,3);
			\draw[-,blue, thick] (-30,2) -- (-25,2);
			\draw [->,blue,thick] (-30,2) -- (-30,3);
			\foreach \x in {1,2,3}	\draw[-,blue,thick] (-25,\x) -- (-7, \x);
			\draw [->,blue,thick] (-25,2) -- (-25,1);
			\draw [->,blue,thick] (23,3) -- (23,4);
			\draw[-,blue,thick] (-7,1) -- (0, 1);
			\draw[-,blue,thick] (-7,3) -- (30, 3);
			\draw[-,blue,thick] (23,4) -- (30,4);
			\foreach \x in {1,3,4}	\draw[-,blue,thick] (30,\x) -- (40, \x);
			\draw [->,dashed,blue, thick] (30,3) -- (30,1);
			\draw [-,thick] (-35,-0.6) -- (-35,-0.4);
			\node at (-35, -0.8){$t_0$};
			\draw [-,thick] (30,-0.6) -- (30,-0.4);
			\node at (30, -0.8){$t_1$};
			\draw [->,thick] (-45,-0.5) -- (45,-0.5);
		\end{tikzpicture}
		\caption{Graphical representation of the Moran model with selection and $N=5$ individuals. Selective arrows are depicted as dashed arrows. Initially, only individual~$3$ is of the beneficial type, shown in blue. Individual $3$ is not replaced at time $t_0$ since individual $5$ is of wild-type. At time $t_1$ the beneficial descendant (individual $4$) of individual $3$ is able to use the selective arrow for reproduction.}
		\label{fig:Moran}
	\end{figure}
	Selection is added in the same manner. For each individual $i$ we introduce another Poisson point processes $S^{i}$ with rate $s_N>0$, again $(S^i, i \in [N])$ are assumed to be independent. At each point $t^{i}$ of $S^{i}$ we sample a uniform child $j\in [N]$ and draw a selective arrow pointing from $i$ to $j$ at time $t^{i}$ (depicted as dashed arrows). If $i$ is of the beneficial type, individual $i$ is allowed to reproduce onto individual $j$ and $j$ inherits the type of individual $i$, otherwise this reproduction event is silent. We introduce selection in this manner, since it allows one to traverse the graph backwards in time and define the potential ancestry of any sample taken at some time, see \cite{Neuhauser1997genealogy}.
	
	The ancestral selection graph (ASG), $\mathcal{A}^N=(\mathcal{A}^N_t, t \geq 0)$, backwards in time is obtained by reversing time in Figure \ref{fig:Moran}, see Figure \ref{fig:ASG} for an illustration. Assume $\mathcal{A}^N_0 \subseteq [N]$ consists of a sample of individuals. Then $\mathcal{A}^N$ traces the potential ancestry of that sample backwards in time. More precisely, whenever we encounter a selective arrow (from an individual not in the ancestry yet), the ancestry branches and also follows the ancestry of the additional individual. Whenever there is a neutral event between two individual in the potential ancestry, the lines coalesce and we just follow the remaining ancestor. This gives rise to a branching and coalescing structure in contrast to classical coalescent processes obtained for neutral populations \cite{Kingman1982}. Let us define $B^N=(B^N_t,t\geq 0)$ as the process counting the number of elements in $\mathcal{A}^N$.
	
	

	\begin{figure}[h]
		\centering	\begin{tikzpicture}[x=1mm,scale=1.2,>=stealth]
			\foreach \x in {0,1,2,3,4} \draw[-] (-40,\x) -- (40,\x);
			\foreach \x in {1,2,3,4,5} \node[left] at (-40,\x-1) {\x};
			\draw [->,thick] (-30,2) -- (-30,3);
			\draw [->,thick] (-25,2) -- (-25,1);
			\draw [->,thick] (-17,0) -- (-17,4);
			\draw [->,thick] (-7,4) -- (-7,2);
			\draw [->,thick] (12,1) -- (12,2);
			\draw [->,thick] (23,3) -- (23,4);
			\draw [->,dashed,thick] (-35,4) -- (-35,2);
			\draw [->,dashed,thick] (30,3) -- (30,1);
			\draw [->,dashed,thick] (15,0) -- (15,2);
			\draw [-,red,thick] (-25,1) -- (40,1);
			\draw [-,red,thick] (-30,3) -- (30,3);
			\draw [->,red,dashed,thick] (30,3) -- (30,1);
			\draw [->,red,dashed,thick] (0,0) -- (0,1);
			\draw [-,red,thick] (-40,0) -- (0,0);
			\draw [->,red,thick] (-25,2) -- (-25,1);
			\draw [->,red,thick] (-30,2) -- (-30,3);
			\draw [-,red,thick] (-40,2) -- (-25,2);
			\draw [->,red,dashed,thick] (-35,4) -- (-35,2);
			\draw [-,red,thick] (-40,4) -- (-35,4);
			\draw [-,thick] (-30,-0.6) -- (-30,-0.4);
			\node at (-30, -0.8){$t_0$};
			\draw [-,thick] (30,-0.6) -- (30,-0.4);
			\node at (30, -0.8){$t_1$};
			\draw [-,thick] (-38,-0.6) -- (-38,-0.4);
			\node at (-38, -0.8){$0$};
			
			\draw [->,thick] (-45,-0.5) -- (45,-0.5);
		\end{tikzpicture}
		\caption{In this picture we are following the potential ancestry of individual $2$ from right to left. At time $t_1$ the process encounters a selective event, hence the ancestral graph splits and individual $4$ is considered a potential ancestor. At time $t_0$ there is coalescence event happening since individual $4$ chooses individual $3$ as a parent, hence there is one line less to follow. In the end individuals $\{1,3,5\}$ are considered potential ancestors of individual $2$. Hence, $\mathcal{A}^N_{0}= \{1,3,5\}$ and $B^N_{0}=3$.}
		\label{fig:ASG}
	\end{figure}
	
	\begin{defi}[Line counting process]
		The process $B^N=(B_t^N, t \geq 0)$ with state space $[N]$ and $B_0^N=n$ counting the number of potential ancestors of a sample of $n$ individuals is called the line counting process of the ASG. The process $B^{N}$ is Markovian and its transition rates in $k \in [N]$ are given as
		\begin{align}
			r(k,k+1) &= k s_N \left(1 - \frac{k}{N} \right), \label{eq:transitions rates ASP}\\
			r(k,k-1) &= \frac{\gamma}{N} \binom{k}{2}.
		\end{align}
	\end{defi}
	A practical interest in the line counting process of the ASG lies in the fact that it is dual to the frequency process  of the forward in time model. Let $X^{N}=(X_t^N, t \geq 0)$ denote the number of wild-type individuals in the Moran model at time $t$ forward in time and let $k,n \leq N$ and $t\geq 0$, then we have the following hypergeometric duality
	\begin{align}
		\EE{ \frac{ X_t^{N} (X_t^{N} -1 ) \cdots (X_t^{N}-(n-1))} {N(N-1) \cdots (N-(n-1))} \big| X_0^{N}=k } = \EE{ \frac{k(k-1)\cdots (k-(B_t^N-1))}{N(N-1) \cdots (N-(B_t^N -1))} \big| B_0^N=n },
	\end{align}
	see \cite{BoeGoPoWa1} for more details in that direction. However, in this work we will analyse the line counting process without further regard to this duality.
	
	\subsection{Logistic Branching Processes} \label{sec:logisitic_branching}
Logistic branching processes are an extension of classical branching processes and can be seen as the random analogue of the deterministic logistic demographic model commonly used in mathematical ecology \cite{Kot2001},
\begin{align}
    \frac{dx}{dt} = r x \left(1- \frac{x}{K}\right), \quad x_0=a,
\end{align}
where $r$ denotes the growth rate and $K$ the carrying capacity. In contrast to classical branching processes, the logistic branching process does not exhibit unlimited exponential growth and can thus be seen as a more realistic model of population growth.
    
The line counting process $B^N$ of the ancestral selection graph belongs to the class of logistic branching processes in continuous-time with a \bluenew{frequency-dependent branching rate}. The logistic branching process as \bluenew{it is commonly defined} in \cite{Lambert2005logistic} does not have a frequency-dependent branching rate and instead assumes that the branching events are independent for each individual. Inspired by this definition of a logistic branching process in \cite{Lambert2005logistic}, we \bluenew{allow} a logistic branching process \bluenew{to have a} frequency-dependent branching rate, but refer to it simply as the logistic branching process.
	
The logistic branching process is defined as a continuous time Markov chain with state space $\N_0$ with the following transitions. Each individual may give birth, die or compete with another individual. Assume there are currently $k$ individuals in the logistic branching process. Then, each individual gives birth independently at rate $\rho h(k)$ for some bounded function $h : \N \to \R^+$ and $\rho>0$. The number of individuals born at each birth event is equal to $j$ with probability $\pi_j$, where $(\pi_j, j \geq 1)$ is such that 
	\begin{align}
		\sum_{j \geq 1} \pi_j = 1, \quad \pi_j \geq 0.
	\end{align}
Thus, $\rho$ represents the total birth rate of an individual, whereas the function $h$ describes the frequency dependence in the branching rate. In addition, each individual dies independently at rate $d\geq 0$. The competition between individuals is modelled as a pairwise interaction between individuals, resulting in an additional death rate per individual of $c(k-1)$, with $c\geq 0$. Let $(X_t, t \geq 0)$ denote the number of individuals alive at time $t$. The transition matrix $Q=(q(i,j)_{i,j \in \N_0})$ of $X$ is given by 
	\begin{align} \label{eq:rates_logistic}
		q(i,j) = \begin{cases}
			\rho i h(i) \pi_{j-i}, \quad & \text{if } i \geq 1 \text{ and } j > i, \\
			d i + c i(i-1), \quad & \text{if } i \geq 1 \text{ and } j = i -1, \\
			-i( d+ \rho h(i) + c(i-1)), \quad & \text{if } i \geq 1 \text{ and } j=i,\\
			0, \quad &\text{otherwise.}
		\end{cases}
	\end{align}
\begin{rem} \label{rem:connection_moran_logistic}
	We obtain the line counting process of the Moran model by choosing $\rho = s_N$, $\pi_1 = 1$, $h(k) = 1-k/N$, $c = \gamma /(2 N)$ and $d=0$.
\end{rem}	

	\subsection{Convergence Results}
	Our main results concern the fluctuations of the line counting process and the logistic branching process around their expectation in stationarity. 
		Both cases lead to the same limiting distribution, see Theorem \ref{thm:logistic_branching} and \ref{thm:Moran}, which turns out to be an \emph{Ornstein-Uhlenbeck} process $Y=(Y_t,t \geq 0)$. We call $Y$ an Ornstein-Uhlenbeck process with parameters $\theta>0$, $\sigma^2>0$ and initial distribution $\nu$, if $Y$ is a solution to the stochastic differential equation
	\begin{align}
		dY_t= - \theta Y_tdt+ \sigma dW_t, \quad Y_0 \overset{d}{=} \nu, \label{eq:def OU}
	\end{align}
	where $(W_t, t\geq 0)$ is a standard Brownian motion independent of $Y_0$.

	\subsubsection{Convergence for the logistic branching process}
		Let us from now on consider a sequence of logistic branching processes $X^N:=(X_t^N, t \geq 0)$, where we set the total branching rate as $ \rho_N$, the density dependent modification as $h_N$, the death rate as $d_N$ and the competition rate as $c_N$, where we assume $\rho_N,c_N,d_N \to 0$ as $N \to \infty$. We introduce the rescaled version $\widetilde{X}^N$ by
		\begin{align}
			\widetilde{X}^N_t := \frac{X_{\rho_N^{-1}t}^N-\mu^X_N}{\sigma^X_N}, \quad t \geq 0,
		\end{align}
		where
		\begin{align}
			\mu_N^X := \bar \pi  \rho_N c_N^{-1}, \quad \sigma^X_N := \sqrt{\mu^X_N},\quad  \text{ with } \bar \pi := \sum_{k \geq 1} k  \pi_k. \label{eq:def_mu,sigma_X}
		\end{align}
		\begin{thm} \label{thm:logistic_branching}
			 Assume that $\sum_{k\geq 1} k \pi_k = \bar{\pi} < \infty$, $\sum_{k\geq 1} k^2 \pi_k = v^2<\infty$ and $\sum_{k\geq 1} k^3 \pi_k< \infty$. In addition we assume the following holds as $N \to \infty$
			\begin{align}
				\rho_N \to 0, \quad \frac{\rho_N}{c_N} \to \infty, \quad \frac{c_N}{d_N} \to \infty. \label{eq:assumption_logistic}
			\end{align}
			Also assume that $h_N$ is such that
			\begin{align}
				|h_N( \mu_N^X + x \sigma^X_N)-1| \leq C |x| o\left(\frac{\sigma^X_N}{\mu^X_N}\right), \quad \text{ as } N \to \infty, \label{eq:assumption h}
			\end{align}
			for all $x \in \R$ and some constant $C<\infty$.
			Then, if $\widetilde{X}_0^N\to \nu$ in distribution, it holds  that
			\begin{align}
				(\widetilde{X}^N_t, t \geq 0) \stackrel{d}{\longrightarrow} (Y_t, t \geq 0), \quad \text{ as } N \to \infty,
			\end{align}
			in $\Dc_\R[0,T]$ for all $0<T<\infty$, where $Y=(Y_t, t \geq 0)$ is an Ornstein-Uhlenbeck process with parameters $\bar \pi$ and $v^2+\bar \pi$, that is, $Y$ is a solution to the SDE
			\begin{align}
				dY_t= - \bar \pi Y_t dt + \sqrt{v^2 +\bar \pi} \, dW_t , \quad Y_0 \stackrel{d}{=} \nu,
			\end{align}
			where $(W_t, t \geq 0)$ is a standard Brownian motion independent of $Y_0$.
		\end{thm}
	Assumption \eqref{eq:assumption_logistic} of Theorem~\ref{thm:logistic_branching} implies that $\mu_N^X\to \infty$, which ensures that the logistic branching process survives for a long time. In addition, assumption \eqref{eq:assumption h} implies that the frequency dependence in the branching rate is negligible in the limit.
	\begin{rem} \label{rem:applicability_logistic}
		In Remark~\ref{rem:connection_moran_logistic} we have seen that the line counting process is a particular instance of a logistic branching process. Unfortunately, we cannot directly apply Theorem~\ref{thm:logistic_branching} to that case, since assumption \eqref{eq:assumption h} is not always valid. In particular, we would require
		\begin{align}
			h_N(\mu^X_N) -1  = \mu^X_N /N = \frac{2}{\gamma} s_N
        \end{align}
        to be at most
        \begin{align}
        o(\sigma^X_N/\mu^X_N) = o(s_N^{-\frac{1}{2}} N^{-\frac{1}{2}}).
		\end{align}
		However, this is only true if $s_N = o(N^{-\frac{1}{3}})$. For that reason we provide in the next section the relevant theorem for the line counting process covering a larger range of $s_N$. 
	\end{rem}	
	\subsubsection{Convergence of the line counting process}
	
	The next theorem concerns the functional limit theorem for $B^N$, as $N \to \infty$. In the Theorem we are dealing with the regime of \emph{strong selection}, that is $s_N \to s >0$ as $N \to \infty$ and with the regime of \emph{moderate selection}, which assumes
	\begin{align}
		s_N \to 0, \quad s_N N \to \infty, \quad \text{ as } N \to \infty. \label{eq:definition moderate selection} 
	\end{align}
	
	It is well known that the stationary distribution $B^N_{eq}$ of the line counting process is binomial with parameters $N$ and $\frac{2 s_N}{2 s_N + \gamma}$ conditioned to be non-zero, see \cite{BoeGoPoWa1} or \cite{Cordero2017common}. This motivates the definition of the rescaled version $\widetilde B^N :=(\widetilde B^N_t, t \geq 0)$ of the line counting process, via
	\begin{align}
		\widetilde B^N_t := \frac{ B^N_{t s_N^{-1}} - \mu^B_N}{\sigma^B_N}, \quad \widetilde B^N_0 = \frac{B^N_{0} - \mu^B_N}{\sigma^B_N}, 
	\end{align}
	where we set
	\begin{align}
		\mu^B_N &:= \frac{2 s_N}{2 s_N + \gamma} N, \quad 
		\sigma_N^B:= \sqrt{\frac{2 s_N}{2 s_N + \gamma} \left( 1- \frac{2 s_N}{2 s_N + \gamma} \right)N}. \label{eq:def mu and sigma}
	\end{align}

	\begin{thm}\label{thm:Moran}
		Let $\widetilde{B}^N$ be the rescaled line counting process of the Moran-ASG with transition rates given in \eqref{eq:transitions rates ASP}. Assume $\widetilde{B}_0^N \to \nu$ in distribution as $N \to \infty$ for some probability measure $\nu$ on $\R$ and consider $s_N$ in the moderate or strong selection regime. Then, as $N \to \infty$ we have 
		\begin{align}
			(\widetilde{B}_t^N, t\geq 0) \overset{d}{\longrightarrow} (Y_t, t \geq 0),
		\end{align}
		in $\Dc_\R[0,T]$ for all $0<T<\infty$, where $\Dc_\R[0,T]$ denotes the Skorokhod space of càdlàg functions from $[0,T]$ to $\R$. Furthermore, $Y=(Y_t, t \geq 0)$ is an Ornstein-Uhlenbeck process with parameters $\theta=1$ and $\sigma^2=2$. That is $Y$ is a solution to the SDE  
		\begin{align}
			dY_t= -  Y_tdt+ \sqrt{2} dW_t, \quad Y_0 \overset{d}{=}\nu, \label{OU: SDGL}
		\end{align}
		where $(W_t, t\geq 0)$ is a standard Brownian motion independent of $Y_0$.
	\end{thm}
    Note that Theorem~\ref{thm:Moran} follows from Theorem~\ref{thm:logistic_branching} whenever $s_N =o(N^{-1/3})$, implying that we need a weak enough selection for the more general theorem to hold true. The reason why we can prove Theorem~\ref{thm:Moran} for the larger range of selection in comparison with Theorem~\ref{thm:logistic_branching} stems from the fact that we have explicit knowledge of the stationary distribution of $B^N$. As one will see in the proofs, the precise definition of $\mu_N^B$ and $\sigma_N^B$ results in some helpful cancellations.
	
	\begin{rem} \label{rem:moran_theorem}
		\begin{itemize}
			\item Consider a discrete-time Cannings model, where each pair coalesces with probability $c_N\sim \sigma^2 N^{d-1}$ for some $d\in [0,1)$ and selective strength $s_N \sim \alpha c_N$ with $\alpha >0$, see for example \cite{Huillet2021}. Under assumptions roughly equivalent to Möhle's condition \cite{Moehle2001} it can be shown that the frequency process forwards in time converges to a Wright-Fisher diffusion with selection if one rescales time by $c_N^{-1}$, for example one can adapt the strategy of proof in \cite{Etheridge2011} Section 3.2. Here, in order to obtain a diffusive limit, it is necessary that $c_N$ and $s_N$ are of comparable strength. These assumptions  stand in contrast to ours, where we obtain a diffusive limit for the line counting process if $s_N$ is such that $Ns_N \to \infty$, equivalently $s_N c_N^{-1} \to \infty$. 
			
			\item More general ancestral selection processes describing the genealogy of populations with highly skewed offspring distributions have been considered for example in \cite{BBDP23,GKP23} and are usually termed $\Lambda$-ASGs. Here, $\Lambda$ refers to a finite measure describing the events where a single individual produces a large number of offspring. It is a natural question to consider the scaling limit of a $\Lambda$-ASG. In light of our work one would expect to obtain a  Lévy driven Ornstein Uhlenbeck process, i.e.
			\begin{align}
				dU_t = - \theta U_t dt + d L_t,
			\end{align}
			where $L_t$ is a spectrally negative Lévy process determined by the measure $\Lambda$. The jumps of the rescaled line counting process stem from the fact that in this class of models, a large number of individuals can coalesce simultaneously. 
			\item The line counting process of the discrete ancestral selection graph (DASG) for Cannings models was introduced in \cite{BoeGoPoWa1}. Similarly, as $B^N$ it counts the number of potential ancestors of a sample, but for the Cannings model. Theorem~\ref{thm:Moran} implies that the fluctuations of the DASG converge on the $s_N$ time-scale in the same sense as in Theorem~\ref{thm:Moran} to an Ornstein-Uhlenbeck process under the additional assumption $s_N \leq N^{ -\frac{2}{3} + \eta }$ for any $\frac{1}{3}>\eta >0$, answering the question raised in Remark 5.11 of \cite{BoeGoPoWa1}. This follows immediately with the help of \cite[Lemma 5.9]{BoeGoPoWa1} which established a coupling between the line counting process of the DASG and $B^N$. However, the entire regime of moderately strong selection as in \cite{BoeGoPoWa2} is not captured by the convergence result presented here, since the coupling Lemma 5.9 does not apply to the case $s_N \geq N^{-\frac{2}{3}+\eta}$.
            \item The case of weak selection, i.e. $s_N N \to \alpha \in (0,\infty)$, is excluded by Theorem~\ref{thm:Moran} and in fact the conclusions of Theorem~\ref{thm:Moran} are no longer true. Instead, by the same method as in the present work one can show that
            \begin{align}
                 (B^N_{tN} , t \geq 0 ) \to (Z_t, t \geq 0)\quad \text{ as } N \to \infty,
            \end{align}
            where convergence holds in distribution in $\Dc_\R[0,T]$ for any $T >0$ and $Z$ is a logistic branching process with birth rate $\alpha$ and competition $\gamma/2$ (compare \eqref{eq:rates_logistic}). This is precisely the process which was used in \cite{PP}, where they studied the fixation probability and time of a Moran model in the domain of weak selection.
        \end{itemize}
	\end{rem}

\subsection{Martingale problems and convergence of generators}
In this section, we present our last result and the main tool for the proofs of Theorem \ref{thm:logistic_branching} and \ref{thm:Moran} and fix some notation for the later part. Proposition~\ref{prop:Convergence Gen} is a slight generalisation of Corollary 8.7 of Chapter 4 in \cite{EthierKurtz1986}. 

Let us recall some basic facts about martingale problems, for a more in depth introduction we refer to \cite{Stroock2006}. Fix $(E,r)$ to be a complete and separable metric space. We denote in the following by $\mathcal{C}_b(E)$ the set of bounded continuous functions on $E$ and $\mathcal{M}_b(E)$ denotes the set of bounded measurable functions on $E$. Let $G$ be a linear operator on $\mathcal{D} \subseteq \mathcal{C}_b(E)$ such that $G: \mathcal{D} \mapsto \mathcal{M}_b(E)$. We say that $(Y,\PP)$ solves the \emph{martingale problem} $(G,\mathcal{D})$ with initial distribution $\nu$, if $Y$ has paths in $\mathcal{D}_E[0,\infty)$ and
\begin{align}
	M_f(t) = f(Y_t) - \int_0^t Gf(Y_s) ds, \qquad Y_0 \overset{d}{=} \nu, \label{eq:martingale problem}
\end{align}
is a $\PP$-martingale with respect to the natural filtration $\mathcal{F}_t=\sigma(Y_s, s \leq t)$. A martingale problem $(G,\mathcal{D})$ is unique if for two solutions $(Y,\PP)$ and $(X,\Q)$ with $Y_0 \overset{d}{=} X_0 \overset{d}{=} \nu$ it holds
\begin{align}
	\PP( Y \in B) = \Q( X \in B), \quad \text{ for all Borel } B \in D_E[0,\infty).
\end{align}

\begin{prop}
	\label{prop:Convergence Gen}
	Let $(E,r)$ be a complete and separable metric space and let $E_n \subseteq E$ be some metric space for each $n\in \N$. Let $Y^n$ be a sequence of $E_n$-valued Markov processes with paths in $\mathcal{D}_E[0,\infty)$ and generator $G_n:\mathcal{D}_n \subseteq \mathcal{M}_b(E_n)\to \mathcal{M}_b(E_n)$. Assume that $Y^n$ is a solution to the martingale problem $(G_n,\mathcal{D}_n)$. Let $G: \mathcal{D} \subseteq \mathcal{C}_b(E) \to \mathcal{C}_b(E)$ such that the closure of the linear span of $\mathcal{D}$ contains an algebra that separates points. If
	\begin{enumerate} 
		\item $Y_0^n$ converges to $\nu$ in distribution as $n\to \infty$.
		\item The martingale problem $(G,\mathcal{D})$ with initial distribution $\nu$ has at most one solution. 
		\item 
		For each $f \in  \Dc$ there exists a sequence $f_n \in \Dc_n$ such that 
		\begin{align}
			\lim_{n \to \infty} \norm{f_n-f}=0, \quad    \sup_n \norm{f_n}<\infty, \label{eq:assumption uniform convergence f}
		\end{align}
		and for all compact $K\subseteq E$ we have 
		\begin{align}
			\lim_{n \to \infty}  \sup_{y \in E_n \cap K} | Gf(y)-G_nf_n(y)|=0. \label{eq:assumption convergence gen}
		\end{align}
		\item The compact containment condition is fulfilled, i.e. for all $\varepsilon$ there exists a compact $K \subseteq E$ (possibly depending on $T$ and $\eps$), such that for all $n\geq n_0 \in \N$
		\begin{align}
			\PP \left(Y_t^n \in K \text{ for all } t\leq T\right)\geq 1- \varepsilon. \label{eq:assumption compact containment} 
		\end{align}
	\end{enumerate}
	Then there exists a unique solution $Y$ to the martingale problem $(G, \mathcal{D})$ and $Y^n$ converges in distribution to $Y$ in $\mathcal{D}_E[0,T]$ with $0<T<\infty$, where $Y$ has initial distribution $\nu$.
\end{prop}
\begin{rem} 
	This result is in the spirit of \cite{EthierKurtz1986} Corollary~8.7 in Chapter~4, where it is assumed that there exists a sequence of sets $A_n$ fulfilling $\lim_{n \to \infty} \PP(Y_t^n \in A_n, 0 \leq t \leq T)=1$. Then in order to prove convergence of the process $Y^n$ to $Y$ in $\Dc_E[0,T]$ we would need to show
			\begin{align}
				\lim_{n \to \infty} \sup_{y \in A_n} |Gf(y)-G_nf_n(y)|=0,\label{eq:Cond Gen EK}
			\end{align}
   whereas the other conditions of our result correspond to the ones in \cite{EthierKurtz1986}. We stress that the condition \eqref{eq:assumption convergence gen} is not a direct consequence of \eqref{eq:Cond Gen EK}. It is possible to choose $A_n$ compact by considering for example $A_n$ such that for all $m \geq m_0$, $\PP(Y_s^m \in A_n, 0 \leq s \leq T)\geq 1- 1/n$. However, if the process $Y^n$ is not contained in a compact domain, the sequence $A_n$ will be growing and eventually contain $E_n \cap K$ for any $K$ we pick. Therefore, \eqref{eq:Cond Gen EK} implies \eqref{eq:assumption convergence gen} but not immediately vice versa. In particular, the supremum in \eqref{eq:assumption convergence gen} does not depend on $n$ in the sense that the set over which we take the supremum is at most $K$, whereas in the condition of \cite{EthierKurtz1986} the supremum is depending on $n$ as well. Another possible benefit of the formulation here, is the fact that the result in \cite{EthierKurtz1986} only asks for the existence of some collection $\{A_n\}$ fulfilling the stated conditions. However, in practice it might be difficult to find the right sets. It is a subtle difference and the proofs work rather analogously, but in our case this difference is helpful.  
   
   It is hardly possible to verify condition \eqref{eq:Cond Gen EK} for the proof of Theorem~\ref{thm:logistic_branching} or Theorem~\ref{thm:Moran}. In particular, the convergence of generators is no longer valid in regions far away from the centre of attraction $\mu_N^X$ or $\mu_N^B$ as the drift towards these centers grows faster than linear far away from the center. Therefore, the relaxation of \eqref{eq:Cond Gen EK} is needed. 
   A similar result can be found in the lecture notes \cite{kersting2013skript} (in German). The result there is more general in the sense that in addition it covers the case of discrete time Markov chains, whereas in our case we are covering the continuous-time case. The proof in \cite{kersting2013skript} is obtained from scratch without relying on the methods of \cite{EthierKurtz1986}, making use of these methods allows us to shorten the proof. However, we want to stress that the knowledge of the results in \cite{kersting2013skript} has helped us tremendously in order to provide a short proof of the statement. 
			
Another approach to prove Theorem \ref{thm:Moran} or Theorem \ref{thm:logistic_branching} would be to rely on Theorem 8.2 in \cite{Kurtz1981}, which appears reasonable since there one also obtains a Gaussian process as the limiting object. However, one easily verifies that a crucial condition of Theorem 8.2 is violated, namely that the supremum of the jump rates stays bounded. Here, it would also be necessary to rederive a weaker version of Theorem 8.2, in particular allowing larger jump rates in regions where the process has a small probability to be in.
	\end{rem}

\section{Proof of the main results} \label{sec:Proofs}
This section is dedicated to the proofs of our results, as a first step we will present the proof of Proposition \ref{prop:Convergence Gen} and then give the proofs of Theorem~\ref{thm:logistic_branching} and \ref{thm:Moran}. The proof of Proposition~\ref{prop:Convergence Gen} can be read independently of the other proofs.
\subsection{Proof of Proposition~\ref{prop:Convergence Gen}}
\begin{proof} 
	We are making use of Theorem~8.10 and the associated Corollary~8.15, Chapter~4 in \cite{EthierKurtz1986}, which under our assumptions allows us to consider the compact containment condition in \eqref{eq:assumption compact containment} instead of the relative compactness of $(Y^n)_{n \geq 1}$. Additionally, we have to prove that our assumptions imply condition (c') of Theorem 8.10 as well as (8.33) and (8.34) of Corollary~8.6. The statement then follows from Theorem~8.10.
	
	Introduce the short hand notation $g_n:=G_n f_n$ and $g:= Gf$, as in the proof of Corollary~8.7. in \cite{EthierKurtz1986}. We define the stopping time
	\begin{align}
		\tau_n := \inf \{ t>0 : \int_0^t |g_n(Y^n_s)|^2 ds \geq t(\norm{g}^2 +1) \},
	\end{align}
	and the processes $\xi^n$ and $\varphi^n$ by
	\begin{align}
		\xi^n_t=f_n(Y^n_{t \wedge \tau_n}), \quad \varphi^n_t= g_n(Y^n_t) \chi_{ \{\tau_n>t \} }  \chi_{ \{\sigma_n>t \} },
	\end{align}
    where $\sigma_n := \inf \{ t \geq 0 : Y_t^n \notin K\}$ for a compact $K$ such that $\PP(\sigma_n\leq T) \leq \eps$ for some $\eps>0$, which is possible due to \eqref{eq:assumption compact containment}.
    
	Conditions (8.33) and (8.34) of Corollary 8.6 in \cite{EthierKurtz1986} can now be expressed as
	\begin{align}
		\lim_{n \to \infty} \EE{\sup_{t \in \Q \cap [0,T]} |\xi^n_t-f(Y^n_t)|} =0, \quad \sup_{n\geq 1} \EE{ \left(\int_0^T |\varphi^n_s|^p ds \right)^\frac{1}{p} } < \infty, \label{eq:required conditions EK}
	\end{align}
	for some $p\in (1,\infty]$ and any $T<\infty$. Before we prove that these processes fulfil \eqref{eq:required conditions EK} and the conditions from Theorem 8.10 in \cite{EthierKurtz1986} we show that $\lim_{n \to \infty} \pp{\tau_n <T}=0$, for any $T<\infty$. We have
	\begin{align}
		\pp{\tau_n <T}&= \pp{ \int_0^t |g_n(Y^n_s)|^2 ds \geq t(\norm{g}^2 +1), \text{ for some } t \in [0,T] } \\
		&\leq\pp{ \int_0^t |g_n(Y^n_s)|^2 - |g(Y^n_s)|^2 ds \geq t, \text{ for some } t \in [0,T]} \\
		&\leq \pp{ \int_0^t \sup_{0 \leq u \leq T}\left( |g_n(Y^n_u)|^2 - |g(Y^n_u)|^2 \right) ds \geq t, \text{ for some } t \in [0,T]}. \label{eq:bound this prob}
	\end{align}
	Let $A= \{\sigma_n > T \}$, hence $\PP(A) \geq 1- \eps$. On the event $A$ we have
	\begin{align}
		&\pp{ \int_0^t \sup_{0 \leq u \leq T}\left( |g_n(Y^n_u)|^2 - |g(Y^n_u)|^2 \right) ds \geq t, \text{ for some } t \in [0,T]; A } \\
		&\leq \pp{ t \sup_{y \in E_n \cap K} \left( |g_n(y)|^2 - |g(y)|^2 \right) \geq t, \text{ for some } t \in [0,T]; A } \\
		&\leq \chi_{  \{ \sup_{y \in E_n \cap K} \left( |g_n(y)|^2 - |g(y)|^2 \right) \geq 1 \} } \to 0, 
	\end{align}
	by assumption \eqref{eq:assumption convergence gen}. On the complement  $A^c$ we bound the probability in \eqref{eq:bound this prob} simply by $\eps$. Since $\pp{\tau_n <T}$ can be made smaller than any $\eps>0$, it follows $\lim_{n \to \infty} \PP(\tau_n <T )=0$.
	
	Now, \eqref{eq:required conditions EK} follows using \eqref{eq:assumption uniform convergence f} and noting
	\begin{align}
		\lim_{n \to \infty} \EE{\sup_{t \in \Q \cap [0,T]} \left|\xi^n_t-f(Y^n_t)\right|} \leq \lim_{n \to \infty} (\norm{f_n}+\norm{f}) \pp{\tau_n<T} + \norm{f_n -f} =0.
	\end{align}
	In order to verify the second condition in \eqref{eq:required conditions EK} we choose $p=2$. By the definition of $\tau_n$ we have
	\begin{align}
		\sup_{n \geq 1} \EE{ \left(\int_0^T |\varphi^n_s|^2 dx \right)^{\frac{1}{2}}} \leq T^{\frac{1}{2}} ( \norm{g}^2+1)^{\frac{1}{2}} < \infty.
	\end{align}
	It remains to check the conditions (8.51)-(8.54) of Theorem 8.10 in \cite{EthierKurtz1986}. We have
	\begin{align}
		\sup_{n \geq 1} \sup_{s \leq T} \EE{|\xi^n_s|}\leq \sup_{n \geq 1} \norm{f_n} < \infty,
	\end{align}
	by our assumptions. Similarly, to verify (8.52)
    \begin{align}
    \sup_{n \geq 1} \sup_{s \leq T} \EE{|\varphi^n_s|} &\leq \sup_{n \geq 1} \sup_{s \leq T} \EE{|g_n(Y_t^n) \chi_{ \{ \sigma_n >t\} }  | }\leq \sup_{ n \geq 1} \sup_{x \in K} g_n(x) < \infty,    
    \end{align}
    since $g_n\in M_b(E)$ and the fact that $g_n$ converges uniformly to $g$ on $K$ by \eqref{eq:assumption convergence gen}.    
	In order to verify (8.53), for all $k\geq 0$ let $0 \leq t_1<t_2<\dots<t_k\leq t \leq T$ and $h_1,\dots, h_k \in \mathcal{C}_b(E)$
	\begin{align}
		&\lim_{n \to \infty} \EE{\left| (\xi^n_t-f(Y^n_t) ) \prod_{i=1}^k h_i(Y^n_{t_i})  \right|} \leq \lim_{n \to \infty} \EE{|\xi^n_t-f(Y^n_t)| }  \prod_{i=1}^k \norm{h_i} \\
		&\leq \lim_{n \to \infty} \left( \EE{|f_n(Y^n_t)-f(Y^n_t)| \chi_{ \{\tau_n >t \} } } + \pp{\tau_n<t}(\norm{f_n+f}) \right)\prod_{i=1}^k \norm{h_i}=0,
	\end{align}
	where the first term vanishes as $n \to \infty$ by assumption \eqref{eq:assumption uniform convergence f}. Condition (8.54) of Theorem 8.10 follows similarly, we have
    \begin{align}
        &\EE{\left| (\varphi^n_t-g(Y^n_t) ) \prod_{i=1}^k h_i(Y^n_{t_i})  \right|} \leq  \EE{\left| \varphi^n_t-g(Y^n_t) \right|} \prod_{i=1}^k \norm{h_i} \\
        & \leq \left( \EE{\left| g_n(Y^n_t)\chi_{\{ \tau_n >t \} } -g(Y^n_t) \right| ; \sigma_n >t} + \EE{\left|g(Y^n_t) \right| ; \sigma_n <t} \right)  \prod_{i=1}^k \norm{h_i}  \\
        &\leq \left( \EE{\left| g_n(Y^n_t)\chi_{\{ \tau_n >t \} } -g(Y^n_t) \right| ; \sigma_n >t} + \eps \norm{g} \right)  \prod_{i=1}^k \norm{h_i}. \label{eq:convergence_zero}
    \end{align}
    Now as $n \to \infty$ the first term converges to $0$ by dominated convergence and assumption \eqref{eq:assumption convergence gen} thus for $n$ large enough the above can be made smaller than any $\eps'>0$ by choosing $\eps$ small enough. Consequently, the term in \eqref{eq:convergence_zero} converges to $0$ as $n \to \infty$ and the claim follows by an application of Corollary~8.15 in \cite{EthierKurtz1986}.
\end{proof}

\subsection{Proof of Theorem~\ref{thm:logistic_branching}}
	In this section we provide the necessary lemmata for the proof of Theorem~\ref{thm:logistic_branching} by an application of Proposition~\ref{prop:Convergence Gen}. We start with the convergence of generators. In this section we drop the superscript on $\mu_N^X$ and $\sigma_N^X$ to lighten the notation a little. 
	\begin{lem} \label{lem:generator_logistic}
		Let $\widetilde{G}_N: \mathcal{D}_N \subset \mathcal{M}_b(E_N) \to  \mathcal{M}_b(E_N) $ denote the generator (or rate matrix) of $\widetilde{X}_N$ and let $G : \mathcal{C}_b^3(\R) \to \mathcal{C}_b(\R)$ be the generator of the Ornstein Uhlenbeck process with parameters $\bar \pi $ and $\bar \pi +  v^2$. Then it holds for all $f\in \mathcal{C}_b^3(\R)$ and compact $K \subset \R$ with $f_N : = f_{|E_N}$ that
	\begin{align*}
		\lim_{N\to \infty}\sup_{x_N\in E_N\cap K}\left|\widetilde{G}_Nf_N(x_N)-Gf(x_N)\right|=0.
	\end{align*}
	\end{lem}
	
	Recall that the generator $G$ of an Ornstein-Uhlenbeck process with parameters $0 <\theta<\infty$ and $0<\sigma<\infty$ acts on functions $f \in \mathcal{C}_b^2(\R)$, such that for all $x \in \R$ 
	\begin{align*}
	Gf(x)= - \theta xf^\prime (x) +\tfrac{\sigma^2}{2} f^{\prime \prime}(x).
	\end{align*}
	
	\begin{proof}
	
	Choose $K \subseteq \R$ compact and let $x_N \in E_N\cap K$ with $E_N = \{ x \in \R : x = \frac{k-\mu_N}{\sigma_N}, k \in [N] \}$ and as before we denote by $\widetilde{G}_N$ the generator (or rate matrix) of $\widetilde{X}^N$. Choose $f \in \mathcal{C}_b^3(\R)$ arbitrary and recall $f_N := f_{|E_N }$. Then we have 
	\begin{align*}
		\widetilde{G}_Nf_N(x_N)&= \sum_{j \geq 1} r_{\widetilde{X}}\left(x_N,x_N+\tfrac{j}{\sigma_N}\right)\left(f_N\left(x_N+\tfrac{j}{\sigma_N}\right)-f(x_N)\right) \\ &\phantom{=}+ r_{\widetilde{X}}\left(x_N,x_N-\tfrac{1}{\sigma_N}\right)\left(f_N\left(x_N-\tfrac{1}{\sigma_N}\right)-f(x_N)\right),
	\end{align*}
	where the transition rates $r_{\widetilde{X}} := r_{\widetilde{X}^N}$ follow from \eqref{eq:rates_logistic} and for $j\geq 1$ are given by
	\begin{align*}
		r_{\widetilde{X}}\left(x_N,x_N+\tfrac{j}{\sigma_N}\right)&=h_N\left(\mu_N + x_N\sigma_N\right)\left(\mu_N + x_N\sigma_N\right)\pi_j,\\
		r_{\widetilde{X}}\left(x_N,x_N-\tfrac{1}{\sigma_N}\right)&=\rho_N^{-1}\left(d_N(\mu_N + x_N\sigma_N) + c_N(\mu_N + x_N\sigma_N)^2  \right) + \mathcal{O}(c_N \mu_N \rho_N^{-1}).
	\end{align*}  
	By a second order Taylor expansion around $x_N$ and recalling that $x_N$ is bounded, it follows
	\begin{align*}
		\quad \widetilde{G}_N f_N(x_N)=h_N\left(\mu_N + x_N\sigma_N\right)\left(\mu_N+ x_N\sigma_N\right) \sum_{j \geq 1} \pi_j\left[ \tfrac{j}{\sigma_N}f_N^{\prime}(x_N)+\tfrac{j^2}{2\sigma_N^2}f_N^{\prime \prime}(x_N)+\mathcal{O}\left(\tfrac{j}{\sigma_N^3}\right)\right]\\
		+\rho_N^{-1}\left[d_N(\mu_N+ x_N\sigma_N) + c_N(\mu_N+x_N\sigma_N)^2+\mathcal{O}(c_N \mu_N)\right]\left[-\tfrac{1}{\sigma_N}f_N^{\prime}(x_N)+\tfrac{1}{2\sigma_N^2}f_N^{\prime \prime}(x_N)+\mathcal{O} \left(\tfrac{1}{\sigma_N^3}\right)\right]
	\end{align*}
	Recalling $\mu_N= \tfrac{\rho_N}{c_N} \bar \pi$, $\sigma_N = \sqrt{\mu_N}$, (see \eqref{eq:def_mu,sigma_X}), and collecting all terms belonging to the first derivative of $f_N$ we get
	\begin{align*}
		&\tfrac{1}{\sigma_N}f_N^{\prime}(x_N)\bigg[h_N\left(\mu_N + x_N\sigma_N\right)\left(\mu_N+ x_N\sigma_N\right)\bar{\pi} \\
        &\qquad \qquad - \rho_N^{-1}\left[d_N(\mu_N+ x_N\sigma_N) + c_N(\mu_N+x_N\sigma_N)^2+\mathcal{O}(c_N \mu_N)\right] \bigg] \\
        &= \bluenew{\tfrac{1}{\sigma_N}f_N^{\prime}(x_N)\bigg[ (\mu_N + x_N \sigma_N) \bar \pi + o(\sigma_N) - o(c_N \mu_N \rho_N^{-1}) - }\\
        &\qquad\qquad  \bluenew{- \frac{c_N}{\rho_N} \mu_N^2 - 2 \frac{c_N \mu_N x_N \sigma_N}{\rho_N} + \mathcal{O}(c_N \sigma_N^2 \rho_N^{-1}) + \mathcal{O}(c_N \mu_N) )    \bigg] } \\
		\quad&=\tfrac{1}{\sigma_N}f_N^{\prime}(x_N) \left[ x_N \sigma_N \bar \pi - \frac{2}{\rho_N} c_N \mu_N x_N \sigma_N + \mathcal{O}(c_N \mu_N \rho_N^{-1}) +o(\sigma_N) \right] \\
		\quad&= - \bar \pi x_N  f_N^{\prime}(x_N) + \mathcal{O}(\sigma_N^{-1}) +o(1),
	\end{align*}
    \bluenew{where for the first equality we have used \eqref{eq:assumption h} and that $d_N = o(c_N) $, see \eqref{eq:assumption_logistic}.}
	Similarly, we get for the terms involving the second derivative of $f_N$ 
	\begin{align*}
		&\tfrac{1}{2\sigma_N^2}f_N^{\prime \prime}(x_N)
		\bigg[h_N\left(\mu_N + x_N\sigma_N\right)\left(\mu_N+x_N\sigma_N\right)v^2\\
        & \qquad \qquad + \rho_N^{-1}\left[d_N(\mu_N+x_N\sigma_N) + c_N(x_N\sigma_N+\mu_N)^2 + \mathcal{O}(c_N \mu_N)\right]\bigg] \\
		\quad &= \tfrac{1}{2\sigma_N^2}f_N^{\prime \prime}(x_N) \left[ v^2 \mu_N + \frac{c_N}{\rho_N} \mu_N^2 + \mathcal{O}(c_N \rho_N^{-1} \mu_N \sigma_N) + o(\sigma_N)  \right] \\
		\quad &= f_N^{\prime \prime}(x_N) \frac{v^2+ \bar \pi}{2} + \mathcal{O}(\sigma_N^{-1}).
	\end{align*}
	Hence, we arrive at
	\begin{align*}
		\widetilde{G}_N f_N(x_N) = - \bar \pi x_N  f_N^{\prime}(x_N) + f_N^{\prime \prime}(x_N) \frac{v^2+ \bar \pi}{2} + o(1),
	\end{align*}
	which fits the generator of an Ornstein-Uhlenbeck process with parameters $\bar \pi$ and $v^2 + \bar \pi$ up to a small error independently of $x_N$, hence we have shown the convergence of generators.
\end{proof}
As a next step we need to prove the compact containment condition from Proposition~\ref{prop:Convergence Gen}, hence we aim to show that for all $ \eps >0$ and $x\in \R$ there exists a compact set $K\subset \R$ such that for all $N \geq N_0 \in \N$
	\begin{align}
		\PP_x ( \widetilde{X}_t^N \in K, \forall t \leq T) > 1-\eps.
	\end{align}
	In fact we will prove a slightly stronger statement, that is for some $0 \leq \lambda_1 \leq \lambda_2<\infty$, such that for all $x \in [-\lambda_1,\lambda_1]$ there exists some $N_0\in \N$ such that for $N \geq N_0 \in \N$ it holds that
	\begin{align}
		\PP_x ( \widetilde{X}_t^N \in [-\lambda_2,\lambda_2], \forall t \leq T) > 1 -\eps.
	\end{align}
	By the definition of $\widetilde{X}^N$ this is equivalent to showing
	\begin{align}
		\PP_{x_N} ( X_t^N \in [\mu_N - \lambda_2 \sigma_N, \mu_N + \lambda_2 \sigma_N], \forall t \leq \rho_N^{-1} T ) > 1-\eps,
	\end{align}
	for any $x_N \in [\mu_N - \lambda_1 \sigma_N, \mu_N + \lambda_1 \sigma_N] \cap \N$.
	In the following we will use the discrete time embedded Markov chain $M^N=(M^N_k , k\in \N_0)$ of $X^N$ whose transition probabilities for any state $x \in \N$ are those of $X^N$, precisely
	\begin{align}
		p^N(x,y):=\PP_x( M_1^N= y) = \PP_x(X^N_\tau = y), \quad  x,y \in \N, \label{eq:def_embedded_Markov}
	\end{align}
	where $\tau = \inf \{ t \geq 0 : X_0^N \neq X_t^N \}$ denotes the time of the first jump of $X^N$. The proof of Lemma \ref{lem:compact_containment_logistic} relies on two auxiliary results which are proven in Section \ref{sec:auxiliary}.
	\begin{lem} \label{lem:compact_containment_logistic}
		For all $\eps > 0$ there exists some $0 \leq \lambda_1 \leq \lambda_2<\infty$ and for all $x \in [-\lambda_1,\lambda_1]$ there exists some $N_0 \in \N$ such that for $N>N_0$ it holds    
		\begin{align}
			\PP_x ( \widetilde{X}_t^N \in [-\lambda_2,\lambda_2], \forall t \leq T) > 1 -\eps.
		\end{align}
	\end{lem}
	\begin{proof}
		As a first step we estimate the maximal number of transitions $\widetilde{X}^N$ can have while staying in $[-\lambda_2,\lambda_2]$ for a time interval of length $T$. For that reason let $E^N_1,E^N_2,\dots$ be i.i.d. standard exponential random variables with rate 
		\begin{align}
			r_N = \sup_{ x \in [\mu_N- \lambda_2 \sigma_N,\mu_N+ \lambda_2 \sigma_N]\cap \N} \rho_N^{-1} |q(x,x)|,
		\end{align}
        where $q(x,x)$ is as in \eqref{eq:rates_logistic}. Hence, $r_N$ is the maximal jump rate of $\widetilde{X}^N$ in $[-\lambda_2,\lambda_2]$.
		
		For convenience in notation, let $x^+ = \sup \{ [\mu_N- \lambda_2 \sigma_N,\mu_N+ \lambda_2 \sigma_N]\cap \N \}$, then \bluenew{by bounding the transition rates in \eqref{eq:rates_logistic}}
		\begin{align}
			r_N &\leq \rho_N^{-1} \left( \rho_N x^+ \lVert h_N \rVert_\infty  + d_N x^+ + c_N (x^+ )^2 \right) \\
			&\leq C \mu_N (1+ o(1)),
		\end{align}
        for some finite constant $C>0$, \bluenew{where we also used that $c_N (x^+)^2 /\rho_N = \mathcal{O}(c_N\mu_N^2/\rho_N) = \mathcal{O}(\mu_N)$, which follows from \eqref{eq:def_mu,sigma_X} and \eqref{eq:assumption_logistic}}. Therefore, $\widetilde{X}^N$ has more jumps than $m_N$ before time $T$, if
		\begin{align}
			\sum_{i=1}^{m_N} E^N_i \leq T.
		\end{align}
		Hence, we have
		\begin{align}
			\PP \left(  \sum_{i=1}^{m_N} E^N_i \leq T \right) &=  \PP \left( r_N  \sum_{i=1}^{m_N} E^N_i \leq  r_N T \right) \\
			&\leq \PP \left( \Gamma_{m_N} \leq   T  C \mu_N (1+o(1)) \right), 
		\end{align}
		where $\Gamma_{m_N}$ is a gamma distributed random variable with parameters $m_N$ and $1$. Therefore, letting $m_N = \lceil (1+\delta) T  C \mu_N \rceil $ with $\delta >0$, we get by Theorem 5.1.iii) in \cite{Janson2018} that
		\begin{align}
			\PP \left( \Gamma_{m_N} \leq   T  C \mu_N (1+o(1)) \right) \leq \exp\left( - m_N \left( \frac{1}{1+\delta} - 1- \ln \left( \frac{1}{1+\delta} \right) \right) \right) (1+o(1)). \label{eq:estimate_gamma}  
		\end{align}
		Note that the function $y\mapsto y-1-\ln(y)>0$ for all $y\in(0,1)$ and by assumption \eqref{eq:assumption_logistic} we get $m_N \to \infty$. Hence, the probability in \eqref{eq:estimate_gamma} tends to $0$ as $N \to \infty$.
		
		Let us denote by $A_{m_N}$ the event that there are at most $m_N$ many jumps of $\widetilde{X}^N$ in the time interval $[0,T]$, then we have for any $x \in [-\lambda_1,\lambda_1]$ respectively all $x_N \in [\mu_N-\lambda_1 \sigma_N, \mu_N+\lambda_1 \sigma_N] \cap \N$
		\begin{align}
			\PP_x \left( \exists\,  t \leq T : \widetilde{X}_t^N \notin [-\lambda_2,\lambda_2] \right) &\leq \PP_{x_N} \left( \sup_{0 \leq k \leq m_N}  |M_k^N - \mu_N| \geq \lambda_2 \sigma_N \right) + \PP (A_{m_N}^c) \\
			&\leq \frac{2 v^2 m_N}{ (\lambda_2-\lambda_1)^2\sigma_N^2} + \PP (A_{m_N}^c) \\
			&\leq \frac{2 v^2 (1+\delta) T C + 2 v^2}{ (\lambda_2-\lambda_1)^2}  +\PP (A_{m_N}^c) \label{eq:choose_lambda2},
		\end{align}
		where we made use of Lemma \ref{lem:bound_embedded}.
		Now if we let $N$ be large enough, we have seen in \eqref{eq:estimate_gamma} that $\PP (A_{m_N}^c)$ can be made smaller than any $\eps/2$, for all $\eps <0$. In addition, by choosing $\lambda_2$ large enough we can make the first term in \eqref{eq:choose_lambda2} smaller than $\eps/2$, hence proving the Lemma.
	\end{proof}
	
	\begin{lem}\label{lem:uniqueness}
		The martingale problem $(G,\mathcal{C}_b^3(\R))$ with $Gf(x)= - \theta x f^{\prime}(x) + \frac{\sigma^2}{2}f^{\prime \prime}(x)$ and initial distribution $\nu$ has for any $0 < \theta, \sigma< \infty$ at most one solution.
	\end{lem}

    \begin{proof}[Proof of Lemma~\ref{lem:uniqueness}]
	The proof follows by classical theory, for instance one observes that $G$ matches the generator from Theorem 24.1 in \cite{RogersWilliams2000} with $b(x)=-\theta x$ and $a(x)=\sigma^2$ and fulfils the conditions therein. Hence, the martingale problem is well-posed and therefore has at most one solution.
\end{proof}
	
	\begin{proof}[Proof of Theorem~\ref{thm:logistic_branching}]
			The proof of Theorem~\ref{thm:logistic_branching} is now a direct consequence of Proposition~\ref{prop:Convergence Gen}. The four main conditions of Proposition~\ref{prop:Convergence Gen} (except the first one) are shown to hold true in Lemma~\ref{lem:compact_containment_logistic} and \ref{lem:generator_logistic} as well as Lemma \ref{lem:uniqueness}. The first condition of Proposition~\ref{prop:Convergence Gen} is assumed in Theorem~\ref{thm:logistic_branching}. Finally, note that $\mathcal{C}_b^3(\R)$ is an algebra that separates points.
	\end{proof}
	
	\subsection{Auxiliary Lemma} \label{sec:auxiliary}
	In this section we prove some auxiliary results on the embedded Markov chain $M^N$ as defined in \eqref{eq:def_embedded_Markov}.
	\begin{lem} \label{lem:drift_eta} Under the assumptions of Theorem~\ref{thm:logistic_branching}. Let $\Delta^N_x := \EE{M^N_1 -x \mid M^N_0 =x}$ be the expected drift of $M^N$ in $x$. Then there exists a constant $\eta >0$ such that for all $N$ large enough we have
		\begin{align}
			\Delta^N_x < 0 , \forall \, x \geq \mu_N + \eta \sigma_N, \quad \text{and} \quad  \Delta^N_x > 0 , \forall \, x \leq \mu_N - \eta \sigma_N.
		\end{align}
	\end{lem}
	\begin{proof}
		The proof is just a simple calculation, we only show $\Delta^N_x < 0 , \forall \, x \geq \mu_N + \eta \sigma_N$, the other part follows analogously. Let $x$ be of the form $x = \mu_N + x' \sigma_N$ such that $x'\geq \eta>0$, where we choose $\eta$ later. Then 
		\begin{align}
			\Delta_x^N &= \sum_{y} p^N(x,y) (y-x) = \frac{1}{|q(x,x)|} \sum_{y} q(x,y)  (y-x) \\
			&= \frac{1}{|q(x,x)|} \left( x h_N(x)\sum_{j \geq 1} \rho_N j \pi_j - xd_N - c_N x(x-1) \right)\\
			&= \frac{x}{|q(x,x)|} \left( \rho_N \bar \pi h_N(x) -d_N -  c_N (x-1) \right).
		\end{align}
		Hence, in order to determine the sign of $\Delta_x^N$ we just consider the term in the brackets. Recalling $\mu_N= \tfrac{\rho_N}{c_N} \bar \pi$ (see \eqref{eq:def_mu,sigma_X}) and assumption \eqref{eq:assumption h}, we estimate 
		\begin{align}
			\rho_N \bar \pi h(x) -d_N -  c_N (x-1) &\leq \mu_N c_N h(x) - c_N( \mu_N + \eta \sigma_N) \\
			&\leq \mu_N c_N ( 1+ C |x'| o( \sigma_N/\mu_N)) - c_N( \mu_N + \eta \sigma_N) \\
			&= - \eta  c_N \sigma_N + o(c_N \sigma_N) < 0,
		\end{align}
		for $N$ big enough. The other bound follows in the same manner.
	\end{proof}
	
	\begin{lem}\label{lem:bound_embedded}
		Let $x \in [ \mu_N - \lambda \sigma_N, \mu_N + \lambda \sigma_N ] \cap \N$ for $\lambda > 2 \eta$, with $\eta$ from Lemma~\ref{lem:drift_eta}. Then we have for any $\eta' > \lambda$
		\begin{align}
			\PP_x( \sup_{ 0 \leq k \leq n} |M^N_k-\mu_N| > \eta' \sigma_N) \leq \frac{2v^2 }{( \eta' - \lambda)^2} \frac{n}{\sigma_N^2}. \label{eq:prob_to_estimate}
		\end{align}
	\end{lem}
	\begin{proof}
		The idea of the proof is to bound the desired probability by the probability that an unbiased random walk exits a box. Let $\eta$ be as in Lemma \ref{lem:drift_eta} and we define the following sets
		\begin{align}
			K_c = [\mu_N - c \sigma_N, \mu_N +c \sigma_N ] \cap \N, \, \text{ for } c > 0.
		\end{align}
		In $K_\eta$ it is difficult to estimate the behaviour of $M^N$ (as seen in Lemma \ref{lem:drift_eta}), which is the reason why we work in the bigger boxes $K_c$ for $c > \eta$. \bluenew{As a next step we do a case distinction for the initial value $x\in K_\lambda$ and distinguish the cases i) $x \in K_\eta$, ii) $x \in K_{2 \eta} \setminus K_\eta$, iii) $x \in K_\lambda \setminus K_{2\eta}$.} 
		
		\bluenew{\textbf{Case i)}} If $x \in K_\eta$, we wait until $M^N$ leaves $K_\eta$, where on the event $H_n = \{ |M_k^N-M_{k-1}^N| \leq \eta \sigma_N, \forall \, k \leq n\}$ $M^N$ is in $K_{2 \eta}$ after leaving $K_\eta$. In particular, we can conclude that for $x \in K_\eta$ we have with $x^- = \lfloor \mu_N - 2\eta \sigma_N \rfloor$, $x^+ = \lceil \mu_N + 2\eta \sigma_N \rceil$
		\begin{align}
			&\PP_x(\sup_{0 \leq k \leq n} |M_k^N-\mu_N| > \eta'\sigma_N) \\
			&\quad \leq p_x \PP_{x^+}(\sup_{0 \leq k \leq n} |M_k^N-\mu_N| > \eta'\sigma_N, H_n) + (1-p_x)\PP_{x^-}(\sup_{0 \leq k \leq n} |M_k^N-\mu_N| > \eta'\sigma_N,H_n) \\
            &\qquad + \PP(H_n^c),\label{eq:estimate_random_walk}
		\end{align}
		where $p_x$ denotes the probability that $M^N$ leaves $K_\eta$ upwards. Hence, we consider the case $x\in K_{2 \eta}$  for estimating the probability in \eqref{eq:prob_to_estimate}.
        
		\textbf{Case ii)} \bluenew{If $x\in K_{2 \eta} \setminus K_\eta$}, by the construction of $K_{2 \eta}$, the process $M^N$ has a drift towards $\mu_N$, while being in $K_{2\eta} \setminus K_\eta$. Hence, the probability that $M^N$ leaves the box $[\mu_N\pm \eta' \sigma_N] \cap \N $ can be estimated from above by the probability that an unbiased random walk $(S^N_k, k \geq 0)$ leaves the box $[\mu_N\pm (\eta'-2\eta)\sigma_N]$, where $S^N$ is specified in the next paragraph. $S^N$ is chosen such that whenever $M_N$ is in $K_{j \eta} \setminus K_{2 \eta}$ for $j > 2$ the probability to deviate from $\mu_N$ by more than $\eta' \sigma_N$ is higher for $S_N$ and in addition whenever $M^N$ would enter $K_{2 \eta}$ the process $S^N$ is set to some state $x\notin K_{2 \eta}$, hence $S^N$ stays closer to the boundary of $K_{ \eta'}$.
        

        Let $S^N =(S^N_n, n \geq 0)$ be a random walk with jump distribution $\tilde{p}^N(x,y)$, which is obtained from $p^N(x,y)$ by alternating the probability to jump downwards such that $\tilde{p}^N(x,\cdot)$ is mean zero for all $x\in \N$. \bluenew{Note that $S^N$ jumps down at most $1$, hence we can simply increase or decrease this probability as necessary.} By this construction we ensure that $S^N$ can be coupled to $M^N$ such that whenever both of them are in the same state space $x \notin K_{2 \eta}$, $M^N$ jumps with higher probability in the direction of $\mu_N$.
		
		In addition we ``cut out'' $K_{2 \eta}$ from the state space of $S^N$ by saying that whenever $S^N$ would jump from above into $K_{ 2 \eta}$ we put $S^N$ to the lower boundary of $K_{2\eta}$ (meaning the natural number smaller than $\inf K_{2 \eta}$), hence $S^N$ jumps over $K_{2 \eta}$. In addition, whenever $S^N$ would jump into $K_{2 \eta}$ from below, we know that on $H_n$, $M^N$ can't jump higher than $K_\eta$. Say $S^N$ would jump to a height $d$ above $\inf K_{2\eta}$ we put $S^N$ at height $d + \sup K_{2\eta}$. This construction ensures that $S^N$ is an unbiased random walk on $\Z \cap K_{2\eta}^c$. See also Figure \ref{fig:2} for an illustration.
        \begin{figure}
            \centering
            \includegraphics[width = 0.8 \textwidth, trim = 0 0 0 50]{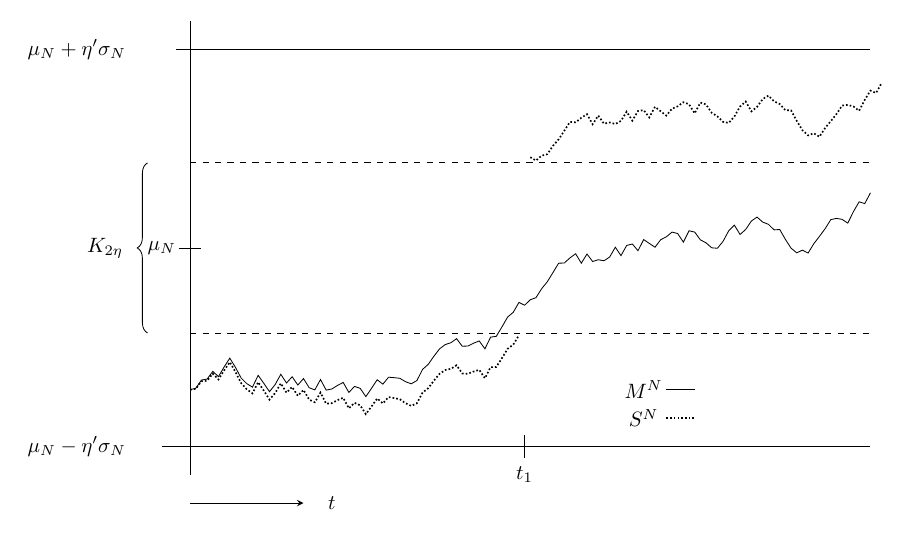}
            \caption{An illustration of the behaviour of the two Markov chains $M^N$ and $S^N$. They both start at the same point, however $M^N$ has a drift towards $\mu_N$. Once $S^N$ would jump into the box $K_{2\eta}$ at time $t_1$ the process jumps over the box. Hence, $S_N$ stays closer to the boundary of $K_{\eta'}$. If then $M^N$ would also leave the box $K_{\eta}$ (not shown in the picture) we still know that $S^N$ is closer to the boundary of $K_{\eta'}$. }
            \label{fig:2}
        \end{figure}
		
		In particular, $S^N$ is a martingale and we can estimate for the first term in \eqref{eq:estimate_random_walk} by Kolmogorov's martingale inequality
		\begin{align}
			\PP_{x^+}(\sup_{0 \leq k \leq n} |M_k^N-\mu_N| > \eta'\sigma_N, H_n) \leq \PP_{x^+}(\sup_{0 \leq k \leq n} |S_k^N-\mu_N| > (\eta'-2\eta)\sigma_N) \leq \frac{\VV{ S^N_n}}{(\eta'-2\eta)^2 \sigma_N^2}.
		\end{align}
		We estimate the variance appearing above by considering the maximal variance of the distribution $\tilde{p}^N(x,\cdot)$ for all $x \in K_{\eta'}$. We have
		\begin{align}
			\sup_{ x \in K_{\eta'}} \sum_y \tilde{p}^N(x,y)(y-x)^2 \leq \sum_{ j \geq 1} j^2 \pi_j =v^2 <\infty.
		\end{align}
		Therefore, we have 
		\begin{align}
			\PP_{x^+}(\sup_{0 \leq k \leq n} |M_k^N-\mu_N| > \eta'\sigma_N, H_n) \leq \frac{ n v^2}{ (\eta'-2\eta)^2} \sigma_N^{-2}, \label{eq:bound_1}
		\end{align}
		where the same bound holds if we start in $x^-$ instead.
		
		Lastly, \bluenew{to finish this case }we have to consider all starting values $x \in K_{2 \eta}, x \notin K_\eta$. Here we know that there exists a drift towards $\mu_N$ but we cannot guarantee that the process does not jump immediately out of $K_{2\eta}$. However, this is no problem at all, since we did not condition $S^N$ to have these small jumps. In particular the probability that $M^N$ escapes $K_{\eta'}$ is smaller than the probability that $S^N$ does in time $n$, hence we obtain the same bound as \eqref{eq:bound_1} with $x$ instead of $x^+$.

        \bluenew{\textbf{Case iii)} Let us consider the remaining case, where $x \in K_\lambda$ but $x \notin K_{2 \eta}$}, as assumed in the Lemma. Here the argument works analogously, we estimate
		\begin{align}
			\PP_x(\sup_{ 0 \leq k \leq n} | M_k^N - \mu_N| > \eta' \sigma_N) \leq \PP_x(\sup_{ 0 \leq k \leq n} | S_k^N - \mu_N| > \eta' \sigma_N) \leq \frac{n v^2}{ (\eta' - \lambda )^2} \sigma_N^{-2}. \label{eq:RW_general bound}
		\end{align}

		Therefore, for all $x \in K_\lambda$ we have with $\lambda > 2 \eta$
		\begin{align}
			\PP_x( \sup_{ 0 \leq k \leq n} |M^N_k-\mu_N| > \eta' \sigma_N) \leq \frac{v^2 }{( \eta' - \lambda)^2} \frac{n}{\sigma_N^2} + \PP (H_n^c).
		\end{align}
		Let $\{X_1,X_2,...\}$ be i.i.d. with distribution $\pi$, then
		\begin{align}
			\PP(H_n^c) \leq \PP( \cup_{ 0 \leq k \leq n} \{ X_j >\eta \sigma_N\}) \leq n \PP(X_1 >  \eta \sigma_N) \leq  \frac{n C}{\eta ^3 \sigma_N^3},
		\end{align}
		for some constant $C< \infty$ by our assumption on $\pi$. Hence, for $N$ big enough we have
		\begin{align}
			\PP_x( \sup_{ 0 \leq k \leq n} |M^N_k-\mu_N| > \eta' \sigma_N) \leq \frac{2v^2 }{( \eta' - \lambda)^2} \frac{n}{\sigma_N^2}.
		\end{align}
		
	\end{proof}

\subsection{Proof of Theorem \ref{thm:Moran} }
Unfortunately, Theorem~\ref{thm:Moran} does not directly follow from Theorem~\ref{thm:logistic_branching}, since for
$s_N \notin o(N^{-\frac{1}{3}})$ assumption \eqref{eq:assumption h} is violated. Therefore, we provide a proof along the same line of arguments for the particular relevant case of the line counting process $B^N$. This proof builds on the fact that we have a more precise understanding of the stationary distribution of $B^N$, which allows us to define $\mu^B_N$ and $\sigma_N^B$ as in \eqref{eq:def mu and sigma}. One can easily check, that the definitions in \eqref{eq:def mu and sigma} and \eqref{eq:def_mu,sigma_X} agree in the first order. However, the precise knowledge of the second order of $\mu^B_N$ and $\sigma^B_N$ results in some helpful cancellations in the proof of Lemma~\ref{lem:generator_moran}. As before, in order to lighten the notation a little in the following we will also drop the superscripts on $\mu_N^B$ and $\sigma_N^B$ and simply write $\mu_N,\sigma_N$.

\begin{lem}\label{lem:generator_moran}
	Let $\widetilde{G}_N: \mathcal{D}_N \subseteq \mathcal{M}_b(E_N) \rightarrow \mathcal{M}_b(E_N) $ denote the generator (or rate matrix)  of $\widetilde{B}^N$ and let $G: \mathcal{C}_b^3(\R)\rightarrow \mathcal{C}_b(\R)$ be the generator of the Ornstein-Uhlenbeck process $Y$ with parameters $1$ and $\sqrt{2}$. For all $f \in \mathcal{C}_b^3(\R)$ and compact $K\subset \R$ with $f_N:=f_{\mid_{E_N}}$, we have
	\begin{align*}
		\lim_{N\to \infty}\sup_{x_N\in E_N\cap K}\left|\widetilde{G}_Nf_N(x_N)-Gf(x_N)\right|=0.
	\end{align*}
\end{lem}

\begin{proof}[Proof of Lemma~\ref{lem:generator_moran}]
	Choose $K \subseteq \R$ compact and let $x_N\in E_N\cap K$ with $E_N = \{ x~\in~\R : x = \frac{k-\mu_N}{\sigma_N}, k \in [N] \}$ and as before we denote by $\widetilde{G}_N$ the generator or rate matrix of $\widetilde{B}^N$. Choose $f \in \mathcal{C}_b^3(\R)$ arbitrary and recall $f_N := f_{|E_N }$. Then we have 
	\begin{align*}
		\widetilde{G}_Nf_N(x_N)&= r_{\widetilde{B}}\left(x_N,x_N+\tfrac{1}{\sigma_N}\right)\left(f_N\left(x_N+\tfrac{1}{\sigma_N}\right)-f(x_N)\right) \\ &\phantom{=}+ r_{\widetilde{B}}\left(x_N,x_N-\tfrac{1}{\sigma_N}\right)\left(f_N\left(x_N-\tfrac{1}{\sigma_N}\right)-f(x_N)\right),
	\end{align*}
	where the transition rates $r_{\widetilde{B}} := r_{\widetilde{B}^N} $ are given by
	\begin{align*}
		r_{\widetilde{B}}\left(x_N,x_N+\tfrac{1}{\sigma_N}\right)&=\left(x_N\sigma_N+\mu_N\right)(1-\tfrac{(x_N\sigma_N+\mu_N)}{N}),\\
		r_{\widetilde{B}}\left(x_N,x_N-\tfrac{1}{\sigma_N}\right)&=\tfrac{ \gamma}{s_N N}\binom{x_N\sigma_N+\mu_N}{2}.
	\end{align*}
	The rates $r_{\widetilde{B}}$ arise from the rates in \eqref{eq:transitions rates ASP} by taking the time rescaling $s_N^{-1}$ into account. A simple calculation yields
	\begin{align*}
		r_{\widetilde{B}}\left(x_N,x_N+\tfrac{1}{\sigma_N}\right)&=\left(1- \tfrac{4s_N}{2s_N+\gamma}\right)x_N\sigma_N+\mu_N-\tfrac{\mu_N^2}{N}+\mathcal{O}\left(\tfrac{\sigma_N^2}{N}\right),\\
		r_{\widetilde{B}}\left(x_N,x_N-\tfrac{1}{\sigma_N}\right)&=\tfrac{2\gamma}{2s_N+ \gamma}x_N\sigma_N +\tfrac{\gamma}{2s_N+ \gamma} \mu_N +\mathcal{O}\left(1+\tfrac{\sigma_N^2}{s_NN}\right),
	\end{align*}
	making use of the fact that $x_N$ is bounded by assumption and that $s_N N \to \infty$. Using a second order Taylor expansion around $x_N$ and the precise value of $\mu_N = \frac{2 s_N}{2 s_N + \gamma} N$ (see \eqref{eq:def mu and sigma}), it follows
	\begin{align*}
		\widetilde{G}_N f_N(x_N)&= \left[ \left(1- \tfrac{4s_N}{2s_N+\gamma}\right)x_N\sigma_N+\mu_N-\tfrac{\mu_N^2}{N}+\mathcal{O}\left(\tfrac{\sigma_N^2}{N}\right)\right]\left[ \tfrac{1}{\sigma_N}f_N^{\prime}(x_N)+\tfrac{1}{2\sigma_N^2}f_N^{\prime \prime}(x_N)+\mathcal{O}\left(\tfrac{1}{\sigma_N^3}\right)\right]\\
		&\phantom{=} +\left[\tfrac{2\gamma}{2s_N+ \gamma}x_N \sigma_N+ \tfrac{\gamma}{2s_N+ \gamma}\mu_N +\mathcal{O}\left(1+\tfrac{\sigma_N^2}{s_N N}\right)\right]\left[-\tfrac{1}{\sigma_N}f_N^{\prime}(x_N)+\tfrac{1}{2\sigma_N^2}f_N^{\prime \prime}(x_N)+\mathcal{O}\left(\tfrac{1}{\sigma_N^3}\right)\right]\\
		&=\tfrac{1}{\sigma_N}f_N^{\prime}(x_N)\left[\left(1-\tfrac{4s_N}{2s_N+\gamma}-\tfrac{2\gamma}{2s_N+\gamma}\right)x_N \sigma_N +\left(1-\tfrac{\gamma}{2s_N+\gamma}\right)\mu_N  -\tfrac{\mu_N^2}{N}\right] \\
		&\phantom{=} + \tfrac{1}{2\sigma_N^2}f_N^{\prime \prime}(x_N)
		\left[ \left(1+\tfrac{\gamma}{2s_N+\gamma}\right)\mu_N -\tfrac{\mu_N^2}{N} \right]+\mathcal{O}\left(\tfrac{1}{\sigma_N}+\tfrac{\sigma_N}{s_NN}\right) \\
		&= -x_N f_N^{\prime}(x_N) + f_N^{\prime \prime}(x_N) + \mathcal{O}\left(\tfrac{1}{\sigma_N}+\tfrac{\sigma_N}{s_NN}\right).
	\end{align*}
	Note that the generator $G$ of the Ornstein-Uhlenbeck process $Y$ at $x_N$ is given by 
	\begin{align*}
		Gf(x_N)= -x_N f_N^{\prime}(x_N) + f_N^{\prime \prime}(x_N).
	\end{align*}
	Hence, it follows
	\begin{align*}
		&\lim_{N\to \infty}\sup_{x_N\in E_N\cap K}\left|\widetilde{G}_Nf_N(x_N)-Gf(x_N)\right|\\
		&= \lim_{N\to \infty}\sup_{x_N\in E_N\cap K} \left|\mathcal{O}\left(\tfrac{1}{\sigma_N}+\tfrac{\sigma_N}{s_NN}\right) \right|\\
		&= 0,
	\end{align*}
	which finishes the proof or Lemma \ref{lem:generator_moran}
\end{proof}

\begin{lem}[Compact Containment Condition]\label{lem:compact_containment_moran}
	Under the assumptions of Theorem~\ref{thm:Moran}. Assume that $\widetilde{B}_0^N \to \nu$ in distribution as $N \to \infty$. Then, for all $\varepsilon >0$ and $T>0$ there exists a compact $K \subset \R$, such that for all $N\geq N_0 \in \N$
	\begin{align*}
		\PP \left(\widetilde{B}_t^N \in K,  t \leq T\right)\geq 1- \varepsilon.
	\end{align*}
\end{lem}

\begin{proof}
	We would like to simply invoke Lemma~\ref{lem:compact_containment_logistic}, but this proof relies on Lemma \ref{lem:drift_eta} and \ref{lem:bound_embedded}. The assumptions of Theorem~\ref{thm:Moran} do not involve the bound \eqref{eq:assumption h} and in fact this assumption does not hold for all choices of $s_N$. However, as one easily checks, only the proof of Lemma~\ref{lem:drift_eta} relies on assumption \eqref{eq:assumption h}. This result is easily obtainable for the line counting process.
	
	Let $M^N$ be the embedded Markov chain in $B^N$ and let $\mu_N,\sigma_N$ be as in \eqref{eq:def mu and sigma}, we have in the notation of Lemma~\ref{lem:drift_eta},
	\begin{align}
		\Delta^N_x = \EE{M_1^N - x \mid M_0^N = x } = x s_N ( 1- x/N) - \gamma /N \binom{x}{2}.
	\end{align}
	Now for $x>\mu_N + \lambda \sigma_N$ for some $\lambda > 0 $ we rewrite
	\begin{align}
     \Delta^N_x =  x \left( s_N ( 1- \frac{x}{N})    - \frac{\gamma x}{2N } + \frac{\gamma }{2N}\right).
	\end{align}
	Hence, we just consider the term in the bracket, plugging in the precise value of $\mu_N$ \bluenew{and using that  $x>\mu_N + \lambda \sigma_N$} we arrive at
	\begin{align}
     s_N \left( 1- \frac{x}{N}\right)    - \frac{\gamma x}{2N } + \frac{\gamma }{2N} & \leq s_N - \frac{2 s_N^2 }{2  s_N + \gamma} - \frac{\gamma s_N}{2s_N + \gamma} - \frac{\lambda \gamma \sigma_N}{2N} + \frac{\gamma}{2 N} \label{eq:cancellation}\\
     &=  - \frac{\lambda \gamma \sigma_N}{2N} + \frac{\gamma}{2 N},
	\end{align}
	Now noting that $\sigma_N = \mathcal{O}(\sqrt{N s_N})$ we get that for $N$ large enough and some $\lambda > 0$ we have that $\Delta^N_x <0$ for all $x \geq \mu_N + \lambda \sigma_N$. Similarly as before, we make use of the precise knowledge of $\mu_N$, resulting in an extremely useful cancellation in \eqref{eq:cancellation}. The other claim of Lemma~\ref{lem:drift_eta} follows analogously. Therefore, we can invoke Lemma~\ref{lem:bound_embedded} as well as \ref{lem:compact_containment_logistic} to prove the compact containment condition for the line counting process.
\end{proof}

\begin{proof}[Proof of Theorem~\ref{thm:Moran}]
		In the same manner, as the proof of Theorem~\ref{thm:logistic_branching} is a direct consequence of Proposition~\ref{prop:Convergence Gen}, the proof of Theorem~\ref{thm:Moran} is a direct consequence of Proposition~\ref{prop:Convergence Gen}. The four main conditions of Proposition~\ref{prop:Convergence Gen} (except the first one) are shown to hold true in Lemma~\ref{lem:compact_containment_moran} and \ref{lem:generator_moran} as well as Lemma \ref{lem:uniqueness}. The first condition of Proposition~\ref{prop:Convergence Gen} is assumed in Theorem~\ref{thm:Moran}. 
	
\end{proof}

\section{Acknowledgement}
We would like to thank two anonymous referees for their helpful comments and Anton Wakolbinger and Götz Kersting for many fruitful discussion and their great support resulting in this work.

\bibliographystyle{plain}
\bibliography{bibfile}

\end{document}